\documentclass[a4paper,reqno]{amsart}
\usepackage{amssymb}
\usepackage{amsmath}
\usepackage{amsthm}
\usepackage{graphicx}
\usepackage{texdraw}
\usepackage{amsfonts}
\usepackage{hyperref}
\usepackage{caption}
\usepackage{color}
\usepackage{enumitem}

\graphicspath{{./pics/}}

 \allowdisplaybreaks \numberwithin{equation}{section}
\begingroup
\theoremstyle{plain}
\newtheorem{theorem}{Theorem}[section]
\newtheorem{proposition}[theorem]{Proposition}
\newtheorem{lemma}[theorem]{Lemma}

\theoremstyle{definition}
\newtheorem{definition}[theorem]{Definition}
\newtheorem{remark}[theorem]{Remark}

\endgroup

\newcommand{\e}{\epsilon}
\newcommand{\ep}{\varepsilon}
\newcommand{\re}{\mathbb R}

\renewcommand{\O}{\Omega }
\newcommand{\p}{\partial }
\def \D {\mathbb D}

\title[Minimization of the Dirichlet-Laplace eigenvalues - diameter constraint]{Minimization of the eigenvalues of the Dirichlet-Laplacian with a diameter constraint}
\author{B. Bogosel, A. Henrot, I. Lucardesi}

\begin{document}

\begin{abstract}
In this paper we look for the domains minimizing the $h$-th eigenvalue of the Dirichlet-Laplacian $\lambda_h$
with a constraint on the diameter. Existence of an optimal domain is easily obtained, and is attained at a constant width body. In the case of a simple eigenvalue, we provide non standard (i.e., non local) optimality conditions. Then we address the question whether or not the disk is an optimal domain in the plane, and we give the precise list of the 17 eigenvalues for which the disk is a local minimum. We conclude by some numerical simulations showing the 20 first optimal domains in the plane.
\end{abstract}

\maketitle

\medskip



\section{Introduction}
Among classical questions in spectral geometry, the problem of minimizing (or maximizing) the eigenvalues
of the Laplace operator with various boundary conditions and various geometric constraints has attracted
much attention since the first conjecture by Lord Rayleigh. In particular, several important open problems
have been solved these last twenty years. We refer e.g. to \cite{H1} and the recent book \cite{H2}
for a good overview on that topic.

Here we consider the eigenvalue problem for the Laplace operator with Dirichlet boundary conditions:
\begin{equation}\label{eigen1}
\left\{\begin{array}{ll}
-\Delta u = \lambda u \quad &  \hbox{in }\O
\\
u=0 \quad &  \hbox{on }\p \O \,.\end{array}
\right.
\end{equation}
In this case, denoting by $0<\lambda_1(\Omega) \leq \lambda_2(\Omega) \ldots$ the sequence of eigenvalues, the relevant problem is the {\it minimization} of $\lambda_h(\Omega)$.

The first constraint that has been considered is the volume one. The fact that the ball minimizes $\lambda_1(\Omega)$
is known as the (Rayleigh-)Faber-Krahn inequality and dates back to the 1920s, see \cite{Fab}, \cite{Kra1}. The second eigenvalue is minimized
by two identical balls: this is the Hong-Krahn-Szeg\"{o} inequality, see e.g. \cite{H1} and the recent \cite{H2} for a short
history of the problem. For the other eigenvalues, we had to wait until 2011-2012 to have a proof of existence of
minimizers, that has been achieved by two different approaches in \cite{Buc} and \cite{Maz-Pra}. The result is
\begin{theorem}[Bucur; Mazzoleni-Pratelli]
The problem
\begin{equation}\label{lb302}
\min \{\lambda_h(\O), \O \subset \mathbb{R}^N, |\O|=c\}
\end{equation}
has a solution. This one is bounded and has finite perimeter.
\end{theorem}
The precise regularity of the minimizer is still unknown, see \cite[chapter 3]{H2} and \cite{BMPV}.

\medskip
For the perimeter constraint, existence and regularity is known, see \cite{dP-V}:
\begin{theorem}[De Philippis-Velichkov]
The problem
\begin{equation}\label{perimlk}
\min\{\lambda_h(\Omega), \Omega \subset \mathbb{R}^N, P(\Omega) \leq P_0\}
\end{equation}
has a solution. Its boundary is $C^{1,\alpha}$ outside a closed set of Hausdorff dimension $N-8$, for every $\alpha \in (0,1)$.
\end{theorem}
Obviously, due to the classical isoperimetric inequality, the ball is still the minimizer for $\lambda_1$. For $\lambda_2$
in two dimensions, the minimizer is a regular convex domain whose boundary has zero curvature exactly at two points,
see \cite{bubuhe}. For higher eigenvalues, general regularity results and qualitative properties of the minimizers are provided in \cite{bogreg}.

\medskip
In this paper, we are interested in the diameter constraint:
\begin{equation}\label{minD}
    \min\{\lambda_h(\Omega), \Omega\subset \mathbb{R}^N, D(\Omega)\leq D_0\}
\end{equation}
where $D(\Omega)$ denotes the diameter of the open set $\Omega$. Obviously the constraint $D(\Omega)\leq D_0$ can be replaced by
$D(\Omega)= D_0$. Existence of a minimizer for \eqref{minD} is easily obtained : since taking the convex hull does not change the diameter,
we can consider minimizing sequences of convex domains which ensure enough compactness and continuity, see Theorem \ref{existheo} below.
Moreover, we can prove that the minimizers are {\bf bodies of constant width}. We investigate more deeply the plane situation.
In particular, we give a complete list of values of $h$ for which the disk can or cannot be a minimizer.
By contrast with the case of the volume constraint, where the disk can be the minimizer only for $\lambda_1$ and $\lambda_3$ (for this
last case, it is still a conjecture), as proved by Amandine Berger in \cite{Ber}, here the list of values of $h$ for which the disk
can be the minimizer is long but finite! More precisely, we prove
\begin{theorem}\label{disks}
The disk is a weak local minimizer of problem \eqref{minD} for the following eigenvalues:
\begin{equation}\label{listdisk}
    \lambda_1\,,\,\lambda_2=\lambda_3\,,\, \lambda_4=\lambda_5\,,\, \lambda_7=\lambda_8\,,\, \lambda_{11}=\lambda_{12}\,,\, \lambda_{16}=\lambda_{17}\,,\,
\lambda_{27}\,,\, \lambda_{33}=\lambda_{34}\,,\, \lambda_{41}=\lambda_{42}\,,\, \lambda_{50}.
\end{equation}
In all the other cases, the disk is not a minimizer.
\end{theorem}
A weak local minimizer is simply a critical point for which the second derivative (of the eigenvalue) is non negative,
see the precise definition and the proof of this theorem in Section \ref{secdisk}.

Let us add a few words about the list of eigenvalues given in \eqref{listdisk}. First of all, the only simple eigenvalue which appears in this list is $\lambda_1$, for which the disk is obviously the global minimizer by the isodiametric inequality. For all the other simple eigenvalues $\lambda_h$, it is not difficult to find a
small deformation of the disk which makes $\lambda_h$ decrease. The case of double eigenvalues is much more intricate and need some precise calculations
and fine properties of the Bessel functions.
It is a little bit surprising to see that the disk is a local minimizer for a complete system of double eigenvalues as, for example, $\lambda_2,\lambda_3$.
Indeed, usually in such a case, small perturbations of a domain with a double eigenvalue make one eigenvalue increase while the other one decrease.
With the diameter constraint it is no longer the case: since the perturbations must preserve the diameter, we have a more rigid situation. As explained in
Section \ref{secexi}, the good way to imagine the possible perturbations consists in staying in the class of domains of constant width.
To conclude, our conjecture is that the list \eqref{listdisk} exactly corresponds to all the cases where the disk is the global minimizer.
This conjecture is supported by the numerical results that we present in Section \ref{secnum}. We perform simulations for $h\leq 20$ and we display the minimizer in the cases where the optimal shape is not a disk. In our computations, the minimizer is a body of constant width which seems regular. The Reuleaux triangle (or other Reuleaux polygons)
does not appear here. Actually, the Reuleaux triangle seems to correspond to a {\it maximizer} of the first eigenvalue, as we will explain in a work in progress.

\section{Existence, optimality conditions}\label{secexi}
\subsection{Existence}
We recall that we are interested in open sets of given diameter which minimize the $k$-th eigenvalue of the Laplace operator with
Dirichlet boundary conditions. First of all we prove existence and first properties of minimizers.
\begin{theorem}\label{existheo}
For any integer $h\geq 1$, the problem \eqref{minD} has a solution. This one is convex and is a body of constant width.
For $h=1$, the solution is the ball.
\end{theorem}
\begin{proof}
Let $\Omega$ be any bounded open set and $\widetilde{\O}$ its convex hull. Since $\O$ and $\widetilde{\O}$ have the same
diameter and $\lambda_h(\widetilde{\O})\leq \lambda_h(\O)$, we can restrict ourselves to the class of convex domains.
If $\O_n$ is a sequence of convex domains of diameter less than $D_0$, we can extract a subsequence which converges
for the Hausdorff distance (by Blaschke selection theorem) to some convex set $\O$ whose diameter is less than $D_0$, because the diameter is lower semi-continuous for the Hausdorff convergence (and actually continuous in the subclass of convex
domains). Moreover, the sequence $\O_n$ $\gamma$-converges to $\O$ (see e.g. \cite[Theorem 2.3.17]{H1}) and therefore
$\lambda_h(\O_n)\to \lambda_h(\O)$. This proves that $\O$ is a minimizer.\\
Let us now assume that $\O$ is not a body of constant width. It means that there is a direction $\xi$ for which the width
of $\O$ in this direction is less than $D_0-\delta$ for some $\delta>0$. By continuity, the width of $\O$ will be less than
$D_0-\delta/2$ for all directions in a neighborhood of $\xi$ on the unit sphere. Therefore we can slightly enlarge $\O$
in all the corresponding directions without changing the diameter, contradicting the minimality of $\O$.\\
At last, let us consider the case $h=1$. For any bounded open set $\O$ of diameter $D_0$, let us introduce $B_0$ the ball of same diameter
and $B^*$ the ball of same volume. The isodiametric inequality states that $|\O|=|B^*| \leq |B_0|$ and therefore $\lambda_1(B^*)\geq \lambda_1(B_0)$, while Faber-Krahn inequality implies $\lambda_1(B^*)\leq \lambda_1(\O)$, and the result follows.
\end{proof}

\begin{remark}
It is possible to give a different proof for the optimality of the ball when $h=1$. In the work of Colesanti \cite{colesanti} it is proved that the first eigenvalue of the Dirichlet-Laplace operator satisfies the following Brunn-Minkowski type inequality: given $K_0,K_1$ two convex bodies in $\Bbb{R}^N$ and $t \in [0,1]$ we have
\begin{equation}
 \lambda_1((1-t)K_0 + tK_1)^{-1/2} \geq (1-t) \lambda_1(K_0)^{-1/2} + t\lambda_1(K_1)^{-1/2}.
\label{BM}
 \end{equation}
Moreover, if equality holds then $K_0$ and $K_1$ are homothetic. 

Now let $K$ be a solution of problem \eqref{minD}, which exists due to arguments stated above. Let $B$ be the ball of diameter $D_0$. It is standard that if $K$ has constant width $D_0$ then $K+(-K)$ is a ball and moreover $1/2K+1/2(-K)  = B$ since convex combinations of bodies of constant width have the same constant width (see for example \cite{spheroforms}). Applying inequality \eqref{BM} we get
\[ \lambda_1(B)^{-1/2} = \lambda_1(1/2K+(1/2)(-K)) ^{-1/2} \geq 0.5 \lambda_1(K)^{-1/2}+0.5 \lambda_1(-K)^{-1/2} = \lambda_1(K)^{-1/2},\]
therefore $\lambda_1(B) \leq \lambda_1(K)$. Moreover we must have equality in \eqref{BM} so $B$ and $K$ are homothetic. This gives another proof that the only solution to problem \eqref{minD} is the ball when $h=1$.
\end{remark}

\subsection{Optimality conditions}\label{paroc}
In this section we consider optimal domains in the plane for sake of simplicity, but the result extends without
difficulty to higher dimension.
It is not so easy to write optimality conditions, since the diameter constraint is very rigid: many deformations of a domain $\O$
of constant width will increase its diameter. Our idea is that the good point of view is to make suitable perturbations of
the support function. Let $f$ denote the support function of the convex set $\O$ seen as a (periodic) function of the angle 
$\theta\in [0,2\pi)$, see \cite{Sch} for definition and properties of the support function. 
Here $\theta$ denotes the angle of the unit exterior vector orthogonal to a support line and for a strictly convex set,
to each $\theta$ corresponds a unique point on the boundary that we will denote $M(\theta)$.
It is well known that $f$ must
satisfy
\begin{equation}\label{supp1}
    f''+f\geq 0 \quad \mbox{in the sense of distributions}\,,
\end{equation}
and, conversely, any non negative function satisfying \eqref{supp1} is the support function of a convex body. Moreover, bodies of constant width $D_0$
are characterized by
\begin{equation}\label{supp2}
    \forall \theta,\ f(\theta)+f(\theta+\pi) =D_0.
\end{equation}
We want to perform perturbations of the optimal domain which preserve the diameter. For that purpose, we consider
perturbations of the kind $f_\ep=f+\ep \phi$ with $\phi$ satisfying
\begin{equation}\label{supp3}
    \forall \theta,\ \phi(\theta)+\phi(\theta+\pi) =0.
\end{equation}
In full generality, we should also consider perturbation which preserve the convexity relation \eqref{supp1}
but, for simplicity we will consider here an optimal domain which is $C^2$ regular: this implies in particular
that its support function satisfies 
\begin{equation}\label{supp4}
    \forall \theta,\ f''(\theta)+f(\theta) \geq \alpha_0>0\,,
\end{equation}
where $\alpha_0$ is the infimum of the radius of curvature. We will denote by $R(\theta)=f''(\theta)+f(\theta)$
the radius of curvature at the point of parameter $\theta$. For a domain of constant width, because of relations
\eqref{supp2} and \eqref{supp3}, it is always less than $D_0$.

Thanks to \eqref{supp4}, any perturbation $f_\ep$
is admissible for $\ep$ small enough. On the numerical simulations shown in Section \ref{secnum}, we can observe
two different properties of the minimizers:
\begin{itemize}
\item the optimal domain seems to be $C^2$ regular (and the radius of curvature is far from zero),
\item the eigenvalue associated to an optimal domain is sometimes simple, sometimes double. More precisely,
the pattern is the following: if the index $h$ corresponds to a simple eigenvalue of the disk, then
the eigenvalue of the optimal domain (which is not the disk, except for $h=1$, see Section \ref{secdisk})
is also simple. If the index corresponds to a pair of double eigenvalues for the disk 
$\lambda_h=\lambda_{h+1}$, then the eigenvalue of the optimal domain is simple for $h$ and double for $h+1$.
\end{itemize}
We can now present the optimality condition satisfied by a regular optimal domain in the case of a simple
eigenvalue:
\begin{theorem}
Let $\Omega$ be a regular minimizer for problem \eqref{minD} in the plane. Let us assume that the corresponding
eigenvalue $\lambda_h(\Omega)$ is simple and let us denote by $u_h$ the corresponding (normalized)
eigenfunction. Then, it satisfies
\begin{equation}\label{op1}
\forall \theta,\  |\nabla u_h(M(\theta))|^2 R(\theta) = |\nabla u_h(M(\theta+\pi))|^2 R(\theta+\pi).
\end{equation}
\end{theorem}
\begin{proof}
Let us denote by $f$ the support function of the optimal domain $\Omega$.
As explained above, we consider
perturbations of the support function $f$ of the kind $f_\ep=f+\ep \phi$ with $\phi$ satisfying
$$
    \forall \theta,\ \phi(\theta)+\phi(\theta+\pi) =0.
$$
In such a way, since $h_\ep$ satisfies
$$f_\ep''+f_\ep \geq 0, \quad f_\ep(\theta +\pi)+f_\ep(\theta)=D_0$$
it is the support function of a convex domain of constant width $D_0$. The perturbation $\phi$
defines a deformation field $V$ whose normal component on the boundary of $\Omega$ is $V.n=\phi(\theta)$.
Since the shape derivative of a simple eigenfunction is given by the Hadamard formulae, see e.g. \cite{HP},
\begin{equation}\label{hadamard}
\lambda_h'(\Omega;V)=-\int_{\partial\Omega} |\nabla u_h|^2 V.n ds,
\end{equation} 
with 
\begin{equation}\label{ds}
ds=(f''+f)(\theta) d\theta =R(\theta) d\theta\,,
\end{equation}
the optimality condition reads
\begin{equation}\label{op2}
\forall \phi \mbox{ satisfying \eqref{supp3} }, \int_{-\pi}^\pi |\nabla u_h|^2 R(\theta) \phi(\theta)\,d\theta
\geq 0.
\end{equation}
Obviously, if $\phi$ satisfies \eqref{supp3}, its opposite $-\phi$ also then \eqref{op2} implies
\begin{equation}\label{op3}
\forall \phi \mbox{ satisfying \eqref{supp3} }, \int_{-\pi}^\pi |\nabla u_h|^2 R(\theta) \phi(\theta)\,d\theta
= 0.
\end{equation}
Condition \eqref{supp3} means that $\phi$ is odd in the sense that its Fourier expansion only contains
odd indexes $c_{2k+1}$. Thus, the orthogonality condition \eqref{op3} means that the function 
$|\nabla u_h|^2 R(\theta)$ is even: its Fourier expansion only contains even indexes $c_{2k}$.
In other terms, it satisfies \eqref{op1}.
\end{proof}

\begin{remark}
This optimality condition \eqref{op1} is not standard since it is not local, in contrast with problem \eqref{lb302}, where it would write $|\nabla u_h|^2=constant$ (if the corresponding eigenvalue was simple) or problem \eqref{perimlk}, where it writes $|\nabla u_h|^2=constant*curvature$ (if the corresponding eigenvalue is simple). Here the optimality condition takes into account two diametric opposite points and relate their curvature with the gradient of the eigenfunction. Let us observe that for the disk, condition \eqref{op1} is always satisfied even for an 
eigenfunction corresponding to a double eigenvalue. This confirms Proposition \ref{critical1} and \ref{critical2} which claim that
the disk is always a critical point.
\end{remark}

\section{Local minimality of the disk}\label{secdisk}

We start by recalling a few standard definitions and fix some notations.

Given a bounded open set $\Omega$ of $\mathbb R^2$ and a smooth vector field $V:\re^2\to \re^2$, the first and second order shape derivatives of $\lambda_h$ at $\Omega$ in direction $V$ are given by the following limits (if they exist):
\begin{eqnarray}
& &\lambda_h'(\Omega;V) :=\displaystyle{\lim_{\e \to 0} \frac{\lambda(\Omega_\e) - \lambda(\Omega)}{\e}\,,}
\\
& &\lambda_h''(\Omega;V)  :=\displaystyle{\lim_{\e \to 0} 2 \, \frac{\lambda(\Omega_\e) - \lambda(\Omega) - \e \lambda'(\Omega,V)}{\e^2}\,,}
\end{eqnarray}
where $\Omega_\e$ is the deformed set $\Omega_\e:=\{ x + \e V(x)\ :\ x\in \Omega\}$.

Here we focus our attention to a particular subclass of sets, that of constant width sets. Accordingly, given $\Omega$ with constant width, we consider only deformation fields $V:\mathbb R^2\to \mathbb R^2$ such that, for every $\e$ small enough, the set $\Omega_\e$ has still constant width.
As already pointed out in \S \ref{paroc}, we can associate to such a $V$ a function $\phi$ defined on $[0,2\pi)$, such that the support function $f_\e$ of $\Omega_\e$ is 
\begin{equation}\label{support1}
f_\e(\theta)=f(\theta) + \e \phi(\theta)\,,
\end{equation}
and $\phi(\theta) + \phi(\theta + \pi)=0$. In particular, the Fourier series of $\phi$ is of the form
\begin{equation}\label{support2}
\phi(\theta)= \sum_{k\geq 0} [ a_{2k+1} \cos((2k+1)\theta) + b_{2k+1}\sin((2k+1)\theta)]=  \sum_{k\in \mathbb Z} c_{2k+1}e^{i(2k+1)\theta}\,,
\end{equation}
with $a_\ell,b_\ell\in \mathbb R$, $c_\ell= (a_\ell - i b_\ell)/2$, and $c_{-\ell}=\overline{c_{\ell}}$  for $\ell \geq 0$ .

\begin{definition}
Let $\Omega$ be a set of constant width. We say that $V:\mathbb R^2 \to \mathbb R^2$ is \textit{admissible } if, for every $\e$ small enough, the set $\Omega_\e:=\{x + \e V(x)\ :\ x\in \Omega\}$ has still constant width. Moreover, we will disregard translations, thus we will always take $c_1=c_{-1}=0$.
We say that $\Omega$ is a {\it critical shape} for $\lambda_h$ if the first order shape derivative vanishes for every admissible deformation, and we say that a critical shape is a {\it weak local minimizer} for $\lambda_h$ if the second order shape derivative is non negative for every admissible deformation. 
\end{definition}

In this section we prove that the disk $\D$ is a critical shape in the class of constant width sets (see Propositions \ref{critical1} and \ref{critical2}) and we determine the sign of the second order shape derivative, characterizing the $h$s for which the disk is/is not a weak local minimizer in the class (see Theorems \ref{wlm1} and \ref{wlm2}). These results imply the main Theorem \ref{disks}.

\medskip

Throughout the paper we will use the following representation of the eigenvalues of the disk: for every $h\in \mathbb N$, $\lambda_h=j_{m,p}^2$, for some $m\geq 0$ and $p\geq 1$, being $j_{m,p}$ the $p$-th zero of the $m$-th Bessel function $J_m$. For the benefit of the reader, we recall in the two tables below the values of these objects for small values of $h$, $m$, and $p$.

\medskip

\begin{center}
\bgroup
\def\arraystretch{1.5}
\begin{tabular}{|c|c|c|c|c|c|}
\hline
{$\boldsymbol{m}$ \ {\bf \textbackslash }\ $\boldsymbol p$ } 
& $\boldsymbol 1$ & $\boldsymbol 2$ & $\boldsymbol 3$ & $\boldsymbol 4$ & $\boldsymbol 5$ 
\\
\hline
$\boldsymbol 0$ & $2.4048$  & $5.5201$ & $8.6537$ & $11.7915$ & $ 14.9309$
\\
\hline
 $\boldsymbol 1$ & $3.8317$ &   $7.0156$   &  $10.1735$ & $13.3237$ & $16.4706$ 
\\
\hline
 $\boldsymbol 2$ & $5.1356$ & $8.4172$ & $11.6198$ & $14.7960$ & $21.1170$
\\
\hline
 $\boldsymbol 3$ & $6.3802$ & $9.7610$ & $13.0152$ &  $16.2235$ & $19.4094$
\\
\hline
 $\boldsymbol 4$  & $7.5883$ & $11.0647$ & $14.3725$ & $17.6160$ & $20.8269$
\\
\hline
 $\boldsymbol 5$  & $8.7715$ & $12.3386$ & $15.7002$ & $18.9801$ & $22.2178$
\\
\hline
 $\boldsymbol 6$ & $9.9361$ & $13.5893$ & $17.0038$ & $20.3208$ & $23.5861$
\\
\hline
 $\boldsymbol 7$ & $11.0864$ & $14.8213$ & $18.2876$ & $21.6415$ & $24.9349$
\\
\hline
 $\boldsymbol 8$ &  $12.2251$ & $16.0378$ & $19.5545$ & $22.9452$ & $26.2668$ 
 \\
 \hline
$\boldsymbol 9$ &  $13.3543$ & $17.2412$ & $20.8070$ & $24.2339$ & $25.5837$ 
 \\
 \hline
\end{tabular}
\egroup
\medskip
\captionof{table}{Some numerical computations of $j_{m,p}$.} \label{table1}
\end{center}

\begin{center}
\bgroup
\def\arraystretch{1.5}
\begin{tabular}{|c|c|c|}
\hline
 $\lambda_1 = j_{0,1}^2$ &  $\lambda_{15}= j_{0,3}^2$ &   $\lambda_{30}= j_{0,4}^2$
\\
\hline
 $\lambda_2 = \lambda_3= j_{1,1}^2$ &  $\lambda_{16}= \lambda_{17}= j_{5,1}^2$ &   $\lambda_{31} = \lambda_{32}= j_{8,1}^2$\\
\hline
 $\lambda_4 = \lambda_5= j_{2,1}^2$ &  $\lambda_{18} = \lambda_{19}= j_{3,2}^2$ &   $\lambda_{33} = \lambda_{34}= j_{5,2}^2$\\
\hline
 $\lambda_6 = j_{0,2}^2$ &  $\lambda_{20} = \lambda_{21}= j_{6,1}^2$ &   $\lambda_{35} = \lambda_{36}= j_{3,3}^2$\\
\hline
 $\lambda_7 = \lambda_8= j_{3,1}^2$ &  $\lambda_{22} = \lambda_{23}= j_{1,3}^2$ &   $\lambda_{37} = \lambda_{38}= j_{1,4}^2$\\
\hline
 $\lambda_9 = \lambda_{10}= j_{1,2}^2$ &  $\lambda_{24} = \lambda_{25}= j_{4,2}^2$ &   $\lambda_{39} = \lambda_{40}= j_{9,1}^2$\\
\hline
 $\lambda_{11} = \lambda_{12}= j_{4,1}^2$ &  $\lambda_{26} = \lambda_{27}= j_{7,1}^2$ &   $\lambda_{41} = \lambda_{42}= j_{6,2}^2$\\
\hline
 $\lambda_{13} = \lambda_{14}= j_{2,2}^2$ &  $\lambda_{28} = \lambda_{29}= j_{2,3}^2$ &   $\lambda_{43} = \lambda_{44}= j_{4,3}^2$\\
\hline
\end{tabular}
\egroup
\medskip
\captionof{table}{Representation $\lambda_h= j_{m,p}^2$ for $h\in \{1,\ldots,44\}$.} \label{table2}
\end{center}

\subsection{Simple eigenvalues of the disk}
We begin with the case of simple eigenvalues, corresponding to $\lambda_h=j_{0,p}^2$, for $p\geq 1$. For brevity, in the following, $\lambda_h$ will be denoted by $\lambda$ and the corresponding eigenfunction $u_h$ simply by $u$.

\bigskip

We start by recalling some results on shape derivatives (cf. \cite[Chapter 5]{HP}). Let 
$l_1:C^\infty(\partial \D)\to \mathbb R$ and $l_2:C^\infty(\partial \D)\times C^\infty(\partial \D)\to \mathbb R$ be the following linear form and bilinear form, respectively:
\begin{align}
 & \displaystyle{l_1(\varphi) = -\int_{\partial \D} |\nabla u|^2 \varphi \,,}\notag
\\
& \displaystyle{ l_2(\varphi, \psi) = \int_{\partial \D} \left[ 2 w_\varphi \frac{\partial w_\psi}{\partial n} + \varphi \psi \left( \frac{\partial u}{\partial n}\right)^2\right]\,,}\label{l2}
\end{align}
where $w_\varphi$ (and, similarly, $w_\psi$) solves
\begin{equation}\label{def-w}
\left\{\begin{array}{lll}
-\Delta w_\varphi= \lambda w_\varphi \quad &\hbox{in }\D
\\
w_\varphi= -  \varphi \frac{\partial u}{\partial n} \quad &\hbox{on }\partial \D
\\
\int_{\D}u w_\varphi = 0\,.
\end{array}
\right.
\end{equation}
Then
\begin{align}
\lambda'(\D;V) & = 
l_1(V\cdot n)\,, \label{l'}
\\
\lambda''(\D;V) & = 
l_2(V\cdot n, V\cdot n) +  l_1( Z)\,,\label{l''}
\end{align}
with 
\begin{equation}\label{def-Z}
Z:=(D_{\tau} n\, V_\tau )\cdot V_\tau - 2 \nabla_\tau (V\cdot n)\cdot V_\tau\,,
\end{equation}
being $V_\tau = V - (V\cdot n)n$, $\nabla_\tau \phi = \nabla \phi - (\partial_n \phi) n$, $D_\tau n = D n - (Dn \,n)\otimes n$.

In case of $V=(V_1(\theta), V_2(\theta))$, since on $\partial \D$ we have $n=(\cos\theta, \sin \theta)$ and $\tau = (-\sin \theta, \cos \theta)$, the function $Z$ defined in \eqref{def-Z} reads
$$
Z= - (V\cdot \tau)^2 - 2  (V'\cdot n) (V\cdot \tau)\,,
$$
so that
\begin{equation}\label{l1(Z)}
l_1(Z) =  \frac{j_{0,p}^2}{\pi}  \int_{0}^{2\pi} [ (V\cdot \tau)^2 + 2  (V'\cdot n) (V\cdot \tau)]\,\mathrm d\theta\,.
\end{equation}

Note that if we perform a translation with a constant vector field $V=(a;b)$, we get
$$
l_1(V\cdot n) = 0 \,,\quad l_2(V\cdot n, V\cdot n) =  - j_{0,p}^2 (a^2 + b^2)\,,
$$
$$
l_1(Z) =  \frac{j_{0,p}^2}{\pi} \int_{0}^{2\pi} ( - a \sin \theta + b \cos \theta)^2\mathrm d\theta = j_{0,p}^2 (a^2 + b^2)\,,
$$
so that the first and second order shape derivatives, as expected,  vanish.

Considering more general deformations, we obtain the following

\begin{proposition}\label{critical1}
The disk is a critical shape of $\lambda$, for every $\lambda$ simple, in the class of constant width sets.
\end{proposition}
\begin{proof}
Let $V$ be an admissible deformation field, namely such such that $\Omega_\e=\D + \e V(\D)$ has constant width. On the boundary $\partial \Omega_\e = (I + \e V)(\partial\D)$, we have
\begin{equation}\label{def-V}
V(\theta):=\left\{\begin{array}{lll}
\phi(\theta) \cos\theta - \phi'(\theta)\sin\theta
\\
\phi(\theta) \sin\theta + \phi'(\theta)\cos\theta\,,
\end{array}
\right.
\end{equation}
where $\phi$ is associated to $V$ as in \eqref{support1}. 

According to \eqref{l'}, since $|\nabla u|$ is constant on $\partial \D$, $V\cdot n  = \phi(\theta)$, and by \eqref{support2} $\phi$ has zero average, we infer that $\lambda'(\D;V)=0$, concluding the proof.
\end{proof}

Let us now consider the second order shape derivative. According to \eqref{l''}, we need to compute $l_2(V\cdot n, V\cdot n)$ and $l_1(Z)$, with $Z$ defined in \eqref{def-Z}. In this case, system \eqref{def-w} reduces to
\begin{equation}\label{3-conditions}
\left\{\begin{array}{lll}
-\Delta w = \lambda w \quad &\hbox{in }\D
\\
w= - \frac{\partial u}{\partial n} \phi \quad &\hbox{on }\partial \D
\\
\int_{\D}u w = 0\,.
\end{array}
\right.
\end{equation}
Consider the basis $\{J_n(j_{0,p}\rho) \cos(n\theta); J_n(j_{0,p}\rho )  \sin(n\theta)\}_{n}$ of eigenfunctions of $\lambda$ in polar coordinates $(\rho;\theta)$: we look for $w$ of the form
$$
w(\rho, \theta) = \sum_{n\geq 1} [A_n \cos(n\theta) + B_n \sin (n\theta)] J_n(j_{0,p} \rho)\,.
$$
In view of the basis chosen, the PDE in \eqref{3-conditions} is readily satisfied. Also the third condition of the zero average in \eqref{3-conditions} follows from the radial symmetry of $u(\rho, \theta) = J_0(j_{0,p}\rho) / (\sqrt{\pi} |J_0'(j_{0,p})|)$.
Imposing the boundary condition we get
$$
A_n=\left\{\begin{array}{lll}\displaystyle{
-\frac{\mathrm{sign}(J_0'(j_{0,p})) j_{0,p}}{ J_{n}(j_{0,p})\sqrt{\pi}} a_{n} }\quad & \hbox{if }n=2k+1\,,\ k\geq 0\\
0 \quad & \hbox{else}\,;
\end{array}
\right.
$$
the same equality holds for $B_n$, with $b_n$ in place of $a_n$.
Therefore, the bilinear form \eqref{l2} is given by
\begin{equation}\label{elle2}
l_2(V\cdot  n, V \cdot n)  =   j_{0,p}^2\sum_{k\geq 0}\left( 1 + 2 j_{0,p} \frac{ J_{2k+1}'(j_{0,p}) }{J_{2k+1}(j_{0,p})}  \right) (a_{2k+1}^2 + b_{2k+1}^2)\,.
\end{equation}
Since $V'\cdot n = 0$ and
$$
V\cdot \tau = \phi'(\theta) = \sum_{k\geq 0}(2k +1) [- a_{2k+1} \sin((2k+1)\theta) + b_{2k+1}\cos((2k+1)\theta)]\,,
$$
formula \eqref{l1(Z)} reads
$$
l_1(Z) = j_{0,p}^2 \sum_{k\geq 0} (2k +1)^2 (a_{2k+1}^2 +  b_{2k+1}^2)\,.
$$

We are now in position to write the second order shape derivative of $\lambda$: setting
$$
P_{N}(x):= 1 + N^2 + 2 x \frac{ J_{N}'(x) }{J_{N}(x)} 
$$
we have
\begin{equation}\label{der2}
\lambda''(\D; V) =  j_{0,p}^2\sum_{k\geq 0} P_{2k+1}(j_{0,p}) (a_{2k+1}^2 + b_{2k+1}^2)\,.
\end{equation}
In order to determine the behavior of $\lambda$ at $\D$ we need to investigate the sign of the coefficients $P_{2k+1}(j_{0,p})$. To this aim, we recall some well-known properties of the Bessel functions:
\begin{align}
& \frac{2 N J_N(x)}{x} = J_{N-1}(x) + J_{N+1}(x)\,,\label{bess1}
\\
& 2 J_N' (x) = J_{N-1} (x) - J_{N+1}(x)\,,\label{bess2}
\\
& x J_N'(x) = N J_N (x) - x J_{N+1}(x)\,,\label{bess3}
\\
& x J_N'(x) = - N J_N(x) + x J_{N-1}(x)\,.\label{bess4}
\end{align}

We are now in a position to prove the following
\begin{theorem}\label{wlm1}
The disk is not a weak local minimizer for any simple eigenvalue $\lambda\neq \lambda_1$ in the class of constant width sets.
\end{theorem}
\begin{proof}
As we have already seen in Proposition \ref{critical1}, the disk is a critical point for this kind of deformations. Thanks to formula \eqref{der2}, it is enough to show that for every $p\geq 2$ there exists $n$ such that $P_{2n+1}(j_0,p)<0$ (the case $p=1$ corresponds to $\lambda_1$).

By combining the properties of the Bessel functions \eqref{bess1} and \eqref{bess2} with $N=1$ and $x=j_{0,p}$, and recalling that $J_0(j_{0,p})=0$, we get
$$
P_1(j_{0,p})=2+ 2 j_{0,p} \frac{ J_1'(j_{0,p}) }{J_1(j_{0,p})} =  0\,.
$$
Again exploiting \eqref{bess1}-\eqref{bess3}, it is easy to derive the recursive formula
$$
P_{N+1}(x) =  N^2  + \frac{4x^2}{(N+1)^2 -  P_N(x)}\,,
$$
which, applied twice, gives
$$
P_3(j_{0,p})= 
\frac{32}{8-j_{0,p}^2}<0
$$
for every $p\geq 2$. 
\end{proof}

\begin{remark} Note that, in view of the estimate
\begin{equation}\label{stimaN}
 x \frac{J_N'}{J_N} \geq N - \frac{2x^2}{2N + 1}\,,
\end{equation}
which is valid for $0<x\leq N+ 1/2$ (cf. \cite[Lemma 11]{Kr}), we obtain the positivity of $P_{2n+1}(j_{0,p})$ for $n$ large enough. In particular, we may find an admissible deformation $V$ such that $\lambda''(\D;V)>0$.
\end{remark}

\subsection{Double eigenvalues of the disk}

\subsubsection{First order shape derivative}
Let now $\lambda:=\lambda_h=\lambda_{h+1}=j_{m,p}^2$ be a double eigenvalue of the disk, for some $m,p\geq 1$. It is known (see, e.g., \cite[Theorem 2.5.8]{H1}) that $\e \mapsto \lambda((I+ \e V)(\D))$ has a directional derivative at $\e =0$, which is given by one of the eigenvalues of the symmetric matrix $M$ with components
$$
M_{11} = -\int_{\partial \D} \left( \frac{\partial u_h}{\partial n}\right)^2 V\cdot n \,,\  M_{12}  = -\int_{\partial \D} \left( \frac{\partial u_{h}}{\partial n}\right)\left( \frac{\partial u_{h+1}}{\partial n}\right) V\cdot n \,,\ 
M_{22} = -\int_{\partial \D}\left( \frac{\partial u_{h+1}}{\partial n}\right)^2 V\cdot n \,.
$$
Recalling the expression of the eigenfunctions
$$
u_h(\rho, \theta) = \sqrt{\frac{2}{\pi}} \frac{J_m(j_{m,p} \rho)}{|J_m'(j_{m,p})|} \cos(m \theta)\,,\quad u_{h+1}(\rho, \theta) = \sqrt{\frac{2}{\pi}} \frac{J_m(j_{m,p} \rho)}{|J_m'(j_{m,p})|} \sin(m \theta)\,,
$$
we infer that
\begin{align*}
& \left(\frac{\partial u_h}{\partial n}\right)^2 =\frac{2}{\pi}j_{m,p}^2 \cos^2(m \theta)\,,\quad \left(\frac{\partial u_{h+1}}{\partial n}\right)^2 =\frac{2}{\pi}j_{m,p}^2\sin^2(m \theta)\,,\\
& \left(\frac{\partial u_h}{\partial n}\right) \left(\frac{\partial u_{h+1}}{\partial n}\right) =\pm \frac{2}{\pi}j_{m,p}^2\cos (m\theta) \sin(m \theta)\quad \hbox{on }\partial \D\,.
\end{align*}
Hence, since $V\cdot n = \phi(\theta)$ is orthogonal to any $\cos(2m\theta)$ and $\sin(2m\theta)$
in $[-\pi,\pi]$, we conclude that the matrix $M$ is identically zero; namely
\begin{proposition}\label{critical2}
The disk is a critical shape of $\lambda$, for every $\lambda$ double, in the class of constant width sets.
\end{proposition}

\subsubsection{Second order shape derivative}
Let us now perform the second order shape derivative. We consider variations $\Omega_\e$ of $\D$ with support function
\begin{equation}\label{support}
f_\e(\theta) = 1 + \e \phi(\theta) + \e^2 \psi(\theta)\,,
\end{equation}
with $\phi$ and $\psi$ such that $\phi(\theta) + \phi(\theta + \pi) = \psi(\theta) + \psi(\theta+\pi)=0$. In particular, for $\e$ small, we can parametrize the boundary $\partial \Omega_\e$ as  $(\rho(\theta,\e),\theta)$, with
$$
\rho(\theta,\e) = 1 + \e \phi(\theta) + \e^2 \left (\psi(\theta)  - \frac{(\phi'(\theta))^2}{2} \right) + o(\e^2) \,.
$$

Adapting the computations done in  \cite{Ber}  for $\lambda(\Omega_\e)|\Omega_\e|$ to our problem, exploiting the developments
$$
|\Omega_\e| = \pi + \e^2 \pi \left(\sum_{\ell=-\infty}^{+\infty} (1-\ell^2)|c_\ell|^2 \right) + o(\e^2)\quad \hbox{and}\quad 
\lambda(\Omega_\e) = j_{m,p}^2 + \frac{\e^2}{2}\lambda'' + o(\e^2)\,,
$$
we obtain the following equality (note that we are interested in the sign of the second order shape derivative):
\begin{equation}\label{derivataseconda}
\begin{split}
\frac{\lambda''(\D;V)}{2 j_{m,p}^2} =  \sum_{\ell=-\infty}^{+\infty} (\ell^2 - 1)|c_\ell|^2 + 2  \sum_{|\ell|\neq m} \left( 1 + j_{m,p} \frac{J_\ell'(j_{m,p})}{J_\ell(j_{m,p})}\right) |c_{m-\ell}|^2 +
\\ + 2 q\left[   -\sum_{\ell=-\infty}^{+\infty} \frac12 \ell(2m-\ell) c_\ell c_{2m-\ell}  +  \sum_{|\ell|\neq m} \left( \frac12 + j_{m,p} \frac{J_\ell'(j_{m,p})}{J_\ell(j_{m,p})}\right) c_{m+\ell}c_{m-\ell} \right]\,,
\end{split}
\end{equation}
where $c_\ell$ are the Fourier coefficients of $\phi$ in the exponential form (see \eqref{support2}), i.e., $c_{\ell}=a_{\ell} + i b_{\ell}$, so that $c_{-\ell}=\overline{c_{\ell}}$ and $c_{2\ell} = 0$. The coefficient $q$ is a complex number of modulus 1, and its product with the term in square brackets is real. Notice that the perturbation $\psi$ does not play any role: indeed, the only relevant term in the development would be the $2m$-th coefficient of its Fourier series, which is zero.

\begin{remark}\label{rem-qm}
The coefficient $q$ (which depends on the deformation chosen) acts as a rotation and can always take two values, one opposite to the other. In other words, the second order shape derivative at $\D$ in direction $V$ is of the form $\lambda''(\D;V) = L_1 \pm |L_2|$, for some $L_i\in \mathbb R$. In particular, for $\lambda=\lambda_h(\D)=\lambda_{h+1}(\D)$, we have
\begin{equation}\label{Tay}
\lambda_{h,h+1}(\Omega_\e) = \lambda + \frac{\e^2}{2} (L_1\pm |L_2|) + o(\e^2)\,.
\end{equation}
More precisely, since by definition the eigenvalues are ordered, the plus sign is associated to $\lambda_{h+1}(\Omega_\e)$, the minus sign to $\lambda_{h}(\Omega_\e)$.
\end{remark}

\subsubsection{Sign of $\lambda''$: the case $m=1$}

As a first computation, we consider the case $m=1$. 
Exploiting the fact that $c_{2\ell} = 0$ and $J_{-n}=(-1)^n J_n$, we get
\begin{align}
& \sum_{\ell=-\infty}^{+\infty} (\ell^2 - 1)|c_\ell|^2  =  \sum_{k\geq 0} (8 k^2 + 8 k) |c_{2k+1}|^2\,,\notag
\\
& \sum_{|\ell|\neq 1} \left( 1 + j_{1,p} \frac{J_\ell'(j_{1,p})}{J_\ell(j_{1,p})}\right) |c_{1-\ell}|^2= \sum_{k\geq 0} \left(  2 + {j_{1,p}}\left(\frac{J_{2k}'(j_{1,p})}{J_{2k}(j_{1,p})} + \frac{J_{2k+2}'(j_{1,p})}{J_{2k+2}(j_{1,p})} \right) \right) |c_{2k+1}|^2\,,\label{equality}
\\
&   -\sum_{\ell=-\infty}^{+\infty} \frac12 \ell(2-\ell) c_\ell c_{2-\ell}  = -\frac12 c_1 c_1 + \sum_{k\geq 1}(4k^2 -1)c_{1+2k}c_{1-2k}\,,\notag
\\
&  \sum_{|\ell|\neq 1} \left( \frac12 + j_{1,p} \frac{J_\ell'(j_{1,p})}{J_\ell(j_{1,p})}\right) c_{1+\ell}c_{1-\ell} = \frac12 c_1 c_1 +  \sum_{k\geq 1} \left( 1 + 2 j_{1,p} \frac{J_{2k}'(j_{1,p})}{J_{2k}(j_{1,p})}\right)c_{1+2k}c_{1-2k}\,.\notag
\end{align}
In particular, we have
\begin{equation}\label{lambdasec}
 \frac{\lambda''(\D;V)}{2 j_{1,p}^2} =   \sum_{k\geq 0} P_{1,p}(k)  |c_{2k+1}|^2 + 2 q\sum_{k\geq 1}Q_{1,p}(k) c_{1+2k}c_{1-2k}  \,,
\end{equation}
where
\begin{align}
P_{1,p}(k) &:= 8 k^2 + 8 k +  4 + 2{j_{1,p}}\left(\frac{J_{2k}'(j_{1,p})}{J_{2k}(j_{1,p})} + \frac{J_{2k+2}'(j_{1,p})}{J_{2k+2}(j_{1,p})} \right)\,,\label{P}
\\
Q_{1,p}(k)&:=4k^2 +  2  j_{1,p} \frac{J_{2k}'(j_{1,p})}{J_{2k}(j_{1,p})}\,.\label{Q}
\end{align}

Before stating the result concerning the sign of $\lambda''$ at $\D$, two technical lemmas are in order.

\begin{lemma}\label{lem-resto}
Let $x>0$, $N\in \mathbb N$. Then 
$$
\frac{J_{N+2}'(x)}{J_{N+2}(y)} - \frac{J_{N}'(x)}{J_{N}(x)} =  \frac{2(N+1)}{x} \left[ \frac{J_{N+1}^2(x)}{ J_{N+2}(x) J_N(x)}  - 1\right]\,.  
$$
\end{lemma}
\begin{proof}
The statement readily follows from \eqref{bess1}, \eqref{bess3}, and \eqref{bess4}, indeed we have
\begin{align*}
 \frac{J_{N+2}'(x)}{J_{N+2}(y)} - \frac{J_{N}'(x)}{J_{N}(x)} & =  
 -  \frac{2(N+1)}{x} + J_{N+1}(x)\left[ \frac{1}{ J_{N+2}(x)}  + \frac{1}{J_N(x)}\right]
\\
&=  -  \frac{2(N+1)}{x} + \frac{J_{N+1}(x)}{J_{N+2}(x) J_N(x)}\left[ J_{N+2}(x)  + J_N(x)\right]
\\
& =  -  \frac{2(N+1)}{x} + \frac{2 (N+1)J_{N+1}^2(x)}{x J_{N+2}(x) J_N(x)}\,.
\end{align*}
\end{proof}

\begin{lemma}\label{lem-pq}
Let $P_{1,p}(k)$ and $Q_{1,p}(k)$ be defined in \eqref{P} and \eqref{Q}, respectively, with $p\geq 1$ and $k\geq 0$ integers. Then the following facts hold:
\begin{itemize}
\item[i)] for $p=1$ we have $P_ {1,1}(0)=Q_{1,1}(0)=Q_{1,1}(1)=0$ and $P_{1,1}(k)$, $Q_{1,1}(k+1)>0$ for every $k\geq 1$;
\medskip
\item[ii)] for every $p\geq 1$ and $k\geq 0$ we have $P_{1,p}(k)=Q_{1,p}(k) + Q_{1,p}(k+1)$;
\medskip
\item[iii)] for $p\geq 2$ we have $P_{1,p}(1)<0$.
\end{itemize}
\end{lemma}

\begin{proof}
Item (ii) follows by direct computation.

Exploiting the properties \eqref{bess3} and \eqref{bess4} of the Bessel functions, we get 
\begin{equation}\label{item-ii}
J_0'(j_{1,p})=0\,,\quad  J_2'(j_{1,p}) = - \frac{2}{j_{1,p}}  J_2(j_{1,p})\,,
\end{equation}
which imply that $P_{1,1}(0)=Q_{1,1}(0)=Q_{1,1}(1)=0$. To conclude the proof of item (i), thanks to (ii), we show that $Q_{1,1}(k)$ is a non decreasing function in $k$: for every $k\geq 1$, we have
\begin{align*}
Q_{1,1}(k+1) - Q_{1,1}(k)  & =  4 + 8 k +  2  j_{1,1}  \left[ \frac{J_{2k+2}'(j_{1,1})}{J_{2k+2}(j_{1,1})} -\frac{J_{2k}'(j_{1,1})}{J_{2k}(j_{1,1})}\right] 
\\ & = 4(2k+1)\frac{J_{2k+1}^2(j_{1,1})}{J_{2k+2}(j_{1,1}) J_{2k}(j_{1,1})}>0\,,
\end{align*}
where in the last equality we have used Lemma \ref{lem-resto} with $x=j_{1,1}$ and $N=2k$; while in the last inequality we have used the fact that, for $N\geq 2$, all the Bessel functions $J_N$ are positive till their first zero $j_{N,1}$, which is greater than $j_{1,1}$.

Let now $p\geq 2$. Again in view of the properties of the Bessel functions recalled in \eqref{bess1}-\eqref{bess4}, we get
$$
  J_3(j_{1,p}) =  \frac{4}{j_{1,p}} J_2(j_{1,p})\,,\quad  J_4(j_{1,p}) =  \frac{24-j_{1,p}^2}{j_{1,p}^2} J_2(j_{1,p})\,, \quad  J_{4}'(j_{1,p}) = \frac{8( j_{1,p}^2- 12)}{j_{1,p}^3}  J_2(j_{1,p})\,.
$$
These equalities, combined with \eqref{item-ii}, give
\begin{equation}\label{pp1}
P_{1,p}(1) =20 + 2j_{1,p} \left(\frac{J_{2}'(j_{1,p})}{J_{2}(j_{1,p})} + \frac{J_{4}'(j_{1,p})}{J_{4}(j_{1,p})} \right) 
 = 16 \left[ 1 + \frac{ j_{1,p}^2- 12}{24-j_{1,p}^2} \right] =   \frac{192}{24-j_{1,p}^2}\,.
\end{equation}
Since $j^2_{1,p}>24$ for every $p\geq 2$, the proof of (iii) is achieved.
\end{proof}

Exploiting these properties on the coefficients $P_{1,p}$ and $Q_{1,p}$, we conclude the following
\begin{theorem}\label{thm-m1}
Let $\lambda$ be a double eigenvalue of the disk of the form $\lambda=j_{1,p}^2$ for some $p\geq 1$. If $\lambda= \lambda_2=\lambda_3$, then the $\D$ is a weak local minimizer in the class of constant width sets. In all the other cases, the disk is not a weak local minimizer.
\end{theorem}

\begin{proof}
Let $\lambda=\lambda_2=\lambda_3=j_{1,1}^2$. In view of Lemma \ref{lem-pq}-(ii), the coefficients $Q_{1,1}(k)$ are non negative, thus, by the Young inequality, we have
\begin{align*}
\left| 2 q_1 \sum_{k\geq 1}Q_{1,1}(k) c_{1+2k}c_{1-2k} \right|  & \leq  \sum_{k\geq 1} Q_{1,1}(k)|c_{1+2k}|^2 + \sum_{k\geq 1} Q_{1,1}(k) |c_{1-2k}|^2 
\\ &= \sum_{k\geq 1} Q_{1,1}(k)|c_{1+2k}|^2 +  \sum_{k\geq 0} Q_{1,1}(k+1) |c_{2k+1}|^2
\\ & =  \sum_{k\geq 0}\Big(  Q_{1,1}(k)+ Q_{1,1}(k+1) \Big)  |c_{2k+1}|^2\,.
\end{align*}
Therefore
$$
\frac{\lambda''(\D;V)}{2 j_{1,p}^2}\geq \sum_{k\geq 0}\Big(  P_{1,1}(k) - Q_{1,1}(k) - Q_{1,1}(k+1)  \Big)   |c_{2k+1}|^2  = 0\,,
$$
indeed $Q_{1,1}(k)+ Q_{1,1}(k+1)=P_{1,1}(k)$ by Lemma \ref{lem-pq}-(i).

In all the other cases, namely when $\lambda=j_{1,p}^2$ for some $p\geq 2$, the deformation $V$ associated to $c_3=1$ and $c_i=0$ for every $i\neq 3$, gives a negative second order shape derivative at $\D$: indeed, in view of \eqref{lambdasec}, we get
$$
\frac{\lambda''(\D;V)}{2j^2_{1,p}} =P_{1,p}(1)\,,
$$
which, by Lemma \ref{lem-pq}-(iii)  is negative for every $p\geq 2$.
\end{proof}

\subsubsection{Sign of $\lambda''$: the case $m\geq 2$}
In this subsection, $m$ will always be a natural number greater than or equal to 2.

Similarly as above, we may write
\begin{equation}\label{split-pr}
\frac{\lambda''(\D;V)}{2j_{m,p}^2} = \sum_{k \geq 0} P_{m,p}(k) |c_{2k+1}|^2+ q \sum_{\ell =-\infty}^{+\infty} R_{m,p}(\ell) c_{m-\ell} c_{m+\ell}\,,
\end{equation}
where
\begin{align}
P_{m,p}(k) & :=
8 k^2 + 8 k + 4 + 2j_{m,p}\left(   \frac{J_{2k+1+m}'(j_{m,p})}{J_{2k+1+m}(j_{m,p})} + \frac{J_{2k+1-m}'(j_{m,p})}{J_{2k+1-m}(j_{m,p})}  \right)\,,\label{def-Pmp}
\\
R_{m,p}(\ell) &:=\ell^2 - m^2 + 1 + 2 j_{m,p} \frac{J_\ell'(j_{m,p})}{J_\ell(j_{m,p})} \quad \hbox{for }\ell \neq \pm m\,.\label{def-Rmp}
\end{align}

By the Young inequality, the second term in the right-hand side of \eqref{split-pr} can be bounded as follows:
\begin{align}
&  q \sum_{\ell =-\infty}^{+\infty} R_{m,p}(\ell) c_{m-\ell} c_{m+\ell}  \geq - \frac12 \sum_{\ell =-\infty}^{+\infty} |R_{m,p}(\ell)| |c_{m-\ell}|^2- \frac12\sum_{\ell =-\infty}^{+\infty} |R_{m,p}(\ell)| |c_{m+\ell}|^2\notag
\\
& = -\sum_{k \geq 0}\Big( |R_{m,p}(2k+1 +m)| + |R_{m,p}(2k+1 -m)|\Big)  |c_{2k+1}|^2\,.\label{Young-R}
\end{align}

\begin{remark}\label{rem-0}
We point out that in the second term of the right-hand side of \eqref{split-pr} the sole non mixed term is $c_m c_m$, which corresponds to $\ell =0$. Therefore, if there was $\overline{k}\neq (m-1)/2$ such that $P_{m,p}(\overline{k})<0$, then we would find a width preserving deformation $V$ for which $\lambda''(\D;V)<0$.
On the other hand, if $P_{m,p}$ and $R_{m,p}$ were always non negative, then the second order shape derivative along any width preserving direction, computed at the disk, would be non negative, since 
\begin{equation}\label{relPR}
R_{m,p}(2k+1 +m) +  R_{m,p}(2k+1 -m)  = P_{m,p}(k)\,.
\end{equation}
\end{remark}

Unfortunately, in general none of these two conditions is satisfied, and the study of the sign of $\lambda''$ deserves a more precise investigation, which is object of Theorem \ref{wlm2}. Before stating the result, we give two technical lemmas.

\begin{lemma}\label{lem-pqm2} For every integer $m\geq 2$ set
\begin{equation}\label{abg}
\alpha_m :=2 \frac{(m^2-4)(m^2-1)}{2m^2 +1}\,,\quad \beta_m := 4 (m-2)(m-1)\,,\quad \gamma_m:=4(m+2)(m+1)\,.
\end{equation}
The coefficients $P_{m,p}$ and $R_{m,p}$ defined in \eqref{def-Pmp} and \eqref{def-Rmp}, respectively, satisfy
\begin{itemize}
\item[i)] for every $m\geq 2$, $P_{m,p}(1)< 0$ if and only if 
$ j_{m,p}^2 < \beta_m $ or $ j_{m,p}^2 >\gamma_m$;
\medskip
\item[ii)] for every $m\geq 9$, $P_{m,p}(2)<0$ when $\beta_m < j_{m,p}^2 < \gamma_m$;
\medskip
\item[iii)] for $m=7$ and $p=1$, $P_{7,1}(k) \geq 0$ for every $k\geq 0$;
\medskip
\item[iv)] for $m=7$ and $p=2$, $R_{7,2}(0)<0$ and $R_{7,2}(2k)\geq 0$ for every $k\geq 1$.
\end{itemize}
\end{lemma}
\begin{proof}
Throughout the proof, for brevity we will adopt the notation $y:=j_{m,p}$.

\medskip

Taking $k=1$ in \eqref{def-Pmp} we obtain
$$
P_{m,p}(1)= 2 \left[10 + y\left(   \frac{J_{3+m}'(y)}{J_{3+m}(y)} + \frac{J_{m-3}'(y)}{J_{m-3}(y)}  \right)\right]\,.
$$

In view of the properties of the Bessel functions \eqref{bess1}-\eqref{bess4}, we get
\begin{align*}
& J_{m+2}(y)= \frac{2(m+1)}{y}J_{m+1}\,,\quad J_{m-2}(y)= \frac{2(m - 1)}{y}J_{m-1}(y)\,,
\\
&
J_{m+3}(y)=\left(\frac{4(m+2)(m+1)}{y^2} - 1 \right)  J_{m+1}(y)\,,\quad J_{m-3}(y)= \left(\frac{4(m-2)(m-1)}{y^2}- 1\right)J_{m-1}(y) \,,
\\
& y J_{m+3}'(y) = -(m+3) J_{m+3}(y) + y J_{m+2}(y)\,,\quad  y J_{m-3}'(y)= y (m-3) J_{m-3}(y) - y J_{m-2}(y)\,,
\end{align*}
so that
\begin{align*}
& y  \frac{J_{m+3}'}{J_{m+3}}(y) = -(m+3)+ 2(m+1) \left(\frac{4(m+2)(m+1)}{y^2}- 1 \right)^{-1},
\\
& y \frac{J_{m-3}'}{J_{m-3}}(y) =  (m-3) -  2(m - 1) \left( \frac{4(m-2)(m-1)}{y^2} - 1\right)^{-1}.
\end{align*}
Thus
$$
P_{m,p}(1)=64 \frac{2 (m^2-4)(m^2-1) - y^2  ( 2m^2  + 1 ) }{[4(m+2)(m+1)-y^2]\,[4(m-2)(m - 1)
-y^2]}\,.
$$
This expression allows us to easily obtain the characterization (i): $P_{m,p}(1)\geq 0$ if and only if
$$
0 < j_{m,p}^2 \leq 2 \frac{(m^2-4)(m^2-1)}{2m^2 +1}    \quad \hbox{or}\quad  4 (m-2)(m-1) < j_{m,p}^2 < 4(m+2)(m+1)\,.
$$

\medskip  As already done at the beginning of the proof, iterating the procedure twice more, we may express $J_{m+5}(y)$ and $J_{m+5}'(y)$ in terms of $J_{m+1}(y)$, and $J_{m-5}(y)$ and $J_{m-5}'(y)$ in terms of $J_{m-1}(y)$, to get
\begin{equation}\label{Jpm5}
y\left(   \frac{J_{5+m}'(y)}{J_{5+m}(y)} + \frac{J_{m-5}'(y)}{J_{m-5}(y)}  \right) = -20 + 8 y^2 \left[ \frac{N_1(y^2)}{D_1(y^2)} - \frac{N_2(y^2)}{D_2(y^2)}\right]\,,
\end{equation}
where $N_i$ and $D_i$ are the following polynomials:
\begin{align*}
& N_1(y^2)=-y^2 (m+2) + 2 (m+3)(m+2)(m+1)\,,
\\
& N_2(y^2)=-y^2 (m-2) + 2 (m-3)(m-2)(m-1)\,,
\\
& D_1(y^2)=y^4 - 12 (m+3)(m+2) y^2 + 16 (m+4) (m+3)(m+2)(m+1)\,,
\\
& D_2(y^2)=y^4 - 12 (m-3)(m-2) y^2 + 16 (m-4) (m-3)(m-2)(m-1)\,.
\end{align*}
Inserting \eqref{Jpm5} in \eqref{def-Pmp} for $k=2$, we get
\begin{equation}\label{Pmp2}
P_{m,p}(2)= 8 \left[4 + y^2 \left(\frac{N_1(y^2)}{D_1(y^2)} - \frac{N_2(y^2)}{D_2(y^2)}\right)\right] = 8 \frac{F(y^2)}{D_1(y^2) D_2(y^2)}\,,
\end{equation}
with
\begin{align*}
F(y^2):= & - (144 m^2 + 264) y^6 + (912 m^4 - 3024 m^2 + 20544) y^4 +
\\ & - (1792 m^6 - 17408 m^4 - 12032 m^2 + 211968) y^2 + 
\\ & + 1024 m^8 - 30720 m^6 + 279552 m^4 - 839680 m^2 + 589824\,. 
\end{align*}
Our goal is to give a sufficient condition on $m$ for the negativity of $P_{m,p}(2)$. First, we notice that both $D_1$ and $D_2$ define parabolas with vertical axis, oriented upward. Therefore, $D_i$ are negative in $(\beta_m,\gamma_m)$ if and only if  $D_i(\beta_m)$ and $D_i(\gamma_m)$ are negative. A direct computation allows to conclude that this is true for $m$ greater than or equal to 9. Finally, it is easy to show that, for $m\geq 9$, $F$ and $F'$ are negative at $\beta_m$ and $\gamma_m$, moreover the derivative $F'$ has no critical point inside $(\beta_m,\gamma_m)$. Therefore $F$ is negative in the whole interval $(\beta_m,\gamma_m)$. This concludes the proof of (ii). 

\medskip Let now $m=7$ and $p=1$. It is easy to prove (e.g., by hand or numerically) that $P_{7,1}(k)$ is non negative for small values of $k$, say for $k$ between 0 and 10. For larger values of $k$, we show that $k\mapsto P_{7,1}(k)$ is increasing: indeed, by combining the definition \eqref{def-Rmp} and Lemma \ref{lem-resto}, we infer that the difference between two subsequent terms reads
\begin{align}
P_{7,1}(k+1)-P_{7,1}(k) & = 16k + 16 + 2 y\left[ \frac{J_{2k+10}'(y)}{J_{2k+10}(y)} - \frac{J_{2k-6}'(y)}{J_{2k+6}(y)}\right]
\notag \\
 &= 4\left[   \frac{(2k+9)J_{2k+9}^2(y)}{J_{2k+10}(y) J_{2k+8}(y)} + \frac{(2k-5)J_{2k-5}^2(y)}{J_{2k-4}(y) J_{2k-6}(y)}\right]\,.\label{proof3}
\end{align}
We recall that every Bessel function $J_h$ is positive on $(0,j_{h,1})$ and the sequence of first zeros $\{j_{h,1}\}_{h\in \mathbb N}$ is increasing. In particular, $J_h(j_{7,1})>0$ whenever $h>7$, so that \eqref{proof3} is positive. This concludes the proof of (iii).

\medskip
Finally, let $m=7$ and $p=2$. As for (iii), by direct computation, it is easy to show that $R_{7,2}(0)<0$ and $R_{7,2}(2k)\geq 0$ for $k$ between 1 and 6. For the subsequent terms, we show that $k\mapsto R_{7,2}(2k)$ is increasing: by applying Lemma \ref{lem-resto} with $N=2k$, we get
\begin{align*}
R_{7,2}(2(k+1))- R_{7,2}(2k) & = 8k+4 + 2y \left[ \frac{J_{2k+2}'(y)}{J_{2k+2}(y)} - \frac{J_{2k}'(y)}{J_{2k}(y)}\right]
= \frac{J_{2k+1}^2(y)}{J_{2k+2}(y)J_{2k(y)}}>0\,,
\end{align*}
for every $k\geq 6$. As for \eqref{proof3}, the last inequality follows by the fact that $J_h>0$ in $(0,j_{h,1})$ for every $h\in \mathbb N$, and $j_{h,1}>j_{7,2}$ for every $h\geq 11$. This concludes the proof of (iv).
\end{proof}

\begin{lemma}\label{m7p12}
Let $m=7$ and $p=1$ or $2$. Then 
\begin{itemize}
\item[i)] there exists a deformation $V$ such that the right-hand side of \eqref{split-pr} is negative for suitable choice of $q$;
\item[ii)]  for every $V$ there exists a choice of $q$ that makes the right-hand side of  \eqref{split-pr} non negative.
\end{itemize}
\end{lemma}
\begin{proof}
In the following, when no ambiguity may arise, we shall omit the subscript $m,p$.

Let us prove (i). For $p=1$, we take $q=1$, and $c_i\in \mathbb R$ for every $i$, $c_i=0 $ $\forall i\neq \pm 5, \pm 9$. Then the right-hand side of  \eqref{split-pr} reads $ P(2) |c_5|^2 + P(4) |c_9|^2 + 2  c_5 c_9$ or equivalently, using a matrix formulation, 
$$
\left(
\begin{array}{cc}
  P(2)  &  R(2)
 \\
   R(2)   & P(4)
\end{array}
\right)\left(
\begin{array}{cc}
c_5
 \\
c_9\end{array}
\right)\cdot (c_5\,,\ c_9)\,.
$$
Since the determinant of the above $2\times 2$ matrix is negative, it is enough to take as $(c_5,c_9)$ an eigenvector corresponding to the negative eigenvalue to conclude the proof.

For $p=2$ it is enough  to take $q=1$, $c_i=0$ for every $i\neq 3$, and $c_3=1$: in this case the right-hand side of  \eqref{split-pr} equals $P(3)+R(0)$ which is negative.

\medskip

Let us now prove (ii). First, we notice that for every deformation, there exist only two values of $q\in \mathbb C$, $|q|=1$ that ensure that the right-hand side of  \eqref{split-pr} is a real number, and they are one opposite to the other, say $q=\pm q^*$. Let $p=1$. Arguing by contradiction, it is easy to see that the expressions corresponding to $q^*$ and $-q^*$ cannot be negative simultaneously: indeed, by adding them we would get that $\sum_{k\geq 0} P(k) |c_{2k+1}|^2$ is negative too, which is absurd, since all the $P_{7,1}(k)$s are non negative (see Lemma \ref{lem-pqm2}-(iii)).

For $p=2$ we cannot use the same trick, since $P_{7,2}(3)<0$. Given a deformation, we consider the complex unit number $q$ such that the right-hand side of  \eqref{split-pr} reads
\begin{equation}\label{q+}
 \sum_{k \geq 0} P(k) |c_{2k+1}|^2+ \left| \sum_{\ell =-\infty}^{+\infty} R(\ell) c_{m-\ell} c_{m+\ell}\right|\,.
\end{equation}
Using the easy bound  $|x+y|\geq |x| - |y|$ and the Young inequality (cf. \eqref{Young-R}), we get
\begin{align*}
&  \left| \sum_{\ell =-\infty}^{+\infty} R(\ell) c_{m-\ell} c_{m+\ell}\right| \geq  |R(0)| |c_7|^2  -  \left| \sum_{\ell \neq 0} R(\ell) c_{m-\ell} c_{m+\ell}\right|
 \\
 & \geq 
( |R(0)| - |R(14)| )  |c_7|^2  -  \sum_{ k \geq 0\,,\ k\neq 3 }\Big( |R(2k+1 +m)| + |R(2k+1 -m)|\Big)  |c_{2k+1}|^2\,.
\end{align*}
Since all the $R(\ell)$ are non negative except from $R(0)$ (see Lemma \ref{lem-pqm2}-(iv)), we infer that expression \eqref{q+} can be bounded from below by
$$
 \sum_{k \geq 0}\Big(  P(k) -  R(2k+1 +m) - R(2k+1 -m)\Big) |c_{2k+1}|^2 =0\,,
$$
where the last equality follows from \eqref{relPR}. This concludes the proof.
\end{proof}

We are now in a position to state the following
\begin{theorem}\label{wlm2}
Let $\lambda$ be a double eigenvalue of the disk of the form $\lambda=j_{m,p}^2$ for some $m\geq 2$, $p\geq 1$. If
$$
\lambda = \lambda_4=\lambda_5\,,\, \lambda_7=\lambda_8\,,\, \lambda_{11}=\lambda_{12}\,,\, \lambda_{16}=\lambda_{17}\,,\, \lambda_{27}\,,\, \lambda_{33}=\lambda_{34}\,,\, \lambda_{41}=\lambda_{42}\,,\, \lambda_{50}\,,
$$
then $\D$ is a weak local minimizer in the class of constant width sets. In all the other cases, the disk in not a weak local minimizer. 
\end{theorem}

\begin{proof}
The proof is divided into several steps, in which we distinguish the following groups of pairs $(m,p)$:
\begin{itemize}
\item[{\it Case 1.}]  $(2,1)$, $(3,1)$, $(4,1)$, $(5,1)$, $(5,2)$, $(6,2)$;
\item[{\it Case 2.}] $(2,p)$, $(4,p)$ for $p\geq2$ and $(5,p)$, $(6,p)$, $(7,p)$ for $p\geq 3$;
\item[{\it Case 3.}] $(3,p)$ for $p\geq 2$;
\item[{\it Case 4.}]  $(6,1)$;
\item[{\it Case 5.}] $(8,p)$ for $p\geq 1$;
\item[{\it Case 6.}] $(m,p)$ for $m\geq 9$ and $p\geq 1$;
\item[{\it Case 7.}] $(7,1)$, $(7,2)$.
\end{itemize}
Note that the family of eigenvalues $\{\lambda_4=\lambda_5\,,\, \lambda_7=\lambda_8\,,\, \lambda_{11}=\lambda_{12}\,,\, \lambda_{16}=\lambda_{17}\,,\,  \lambda_{33}=\lambda_{34}\,,\, \lambda_{41}=\lambda_{42}\}$ corresponds to the pairs listed in Case 1, while  $\lambda_{26}=\lambda_{27}$ and $\lambda_{49}=\lambda_{50}$ correspond to the pairs of Case 7 (cf. Table \ref{table2}). Therefore, in Case 1 we will show the weak local minimality of the disk, in Cases 2 to 6, the non weak local minimality of the disk, and in Case 7 we will discuss the different behavior at $\D$ of the associated double eigenvalues. 

\medskip

\noindent {\it Case 1.} As already pointed out in Remark \ref{rem-0}, if we prove that $R_{m,p}(2k+1\pm m)$ are non negative for every $k$, we readily obtain the weak local minimality of $\D$ for such $\lambda$. A numerical computation shows that $R_{m,p} (2k +  1 \pm m)\geq 0 $ for every $k=1,\ldots,10$.
For the subsequent terms, a sufficient condition is the positive monotonicity with respect to $k$. In order to investigate such property, we compute the difference of two subsequent terms: setting for brevity $N:=2k +  1 \pm m$, we have
\begin{align}
& R_{m,p} (2(k+1) +  1 \pm m) -R_{m,p} (2k +  1 \pm m)  \notag
\\
& = 4 (N + 1) + 2 j_{m,p} \left[  \frac{J_{N+2}'(j_{m,p})}{J_{N+2}(j_{m,p})} - \frac{J_{N}'(j_{m,p})}{J_{N}(j_{m,p})}\right] = \frac{4(N+1) J_{N+1}^2(j_{m,p})}{J_{N+2}(j_{m,p}) J_N(j_{m,p})}\,,   \label{remainder}
\end{align}
where for the last equality we have used Lemma \ref{lem-resto}. For $k \geq 10$ it is easy to verify that $J_{N+2}$ and $J_N$ are both positive in $j_{m,p}$, so that the right-hand side of \eqref{remainder} is positive.

\medskip

\noindent{\it Case 2.} A direct computation shows that $j_{m,p}^2>\gamma_m$, therefore Lemma \ref{lem-pqm2}-(i) gives $P_{m,p}(1)<0$. Hence, since $m\neq 2 k +1 $ when $k=1$ (cf. Remark \ref{rem-0}), the deformation corresponding to $c_3=1$ and $c_i=0$ for every $i\neq 3$ gives a negative second order shape derivative at $\D$ in direction $V$.

\medskip

\noindent{\it Case 3.} Taking a deformation $V$ such that $c_3=1$ and $c_i=0$ for every $i\neq 3$, we infer (see Remark \ref{rem-qm}) that the second order shape derivative at $\D$ in direction $V$ reads, up to a positive multiplicative constant, $P_{3,p}(1)\pm R_{3,p}(0)$. A direct computation gives
$$
P_{3,p}(1)\pm R_{3,p}(0) = \frac{64   \left[  80-19 j_{3,p}^2 \pm (j_{3,p}^2-80)\right] }{(j_{3,p}^2-8)(j_{3,p}^2-80)}= \left\{ \begin{array}{lll}
\displaystyle{- \frac{64\cdot 18 j_{3,p}^2 }{(j_{3,p}^2-8)(j_{3,p}^2-80)}<0}
\\
\ 
\\
\displaystyle{-\frac{64\cdot 20}{(j_{3,p}^2-80)}<0}
\end{array}
\right.
$$
where the last inequalities follow from the estimate $j_{3,p}^2>80$ for every $p\geq 2$.

\medskip

\noindent{\it Case 4.} By direct computation, we get  $P_{6,1}(2)<0$. Hence, as $m\neq 2k+1$ when $m=6$ and $k=2$  (cf. Remark \ref{rem-0}), the deformation corresponding to $c_5=1$ and $c_i=0$ for every $i\neq 5$ gives a negative second order shape derivative at $\D$ in direction $V$.

\medskip

\noindent{\it Case 5.} Since $j_{8,1}^2 < 168 = \beta_8$ and $j_{8,p}^2 > 360 = \gamma_8$ for every $p\geq 3$, by  Lemma \ref{lem-pqm2}-(i), we get $P_{8,p}(1)<0$ for every $p\neq 2$. A direct computation shows that $P_{8,3}(2)<0$. Thus, the deformation with coefficients $c_i=\delta_{i3}$ in case $p\neq 2$ and $c_i=\delta_{i5}$ in case $p=2$ gives a negative second order derivative at $\D$.

\medskip

\noindent{\it Case 6.} Here Lemma \ref{lem-pqm2} gives an exhaustive answer: for every $p\geq 1$, $P_{9,p}(k)<0$ either for $k=1$ or for $k=2$. Thus,  to have a negative second order derivative at the disk, it is enough to chose the deformation corresponding to $c_i=1$ if $i=2k+1$ and $0$ else.

\medskip

\noindent{\it Case 7.} Here $m=7$ and $p=1$ (resp. $p=2$). Recalling Table \ref{table2} this corresponds to the second order shape derivative of $\lambda_{26}$ and $\lambda_{27}$ (resp. $\lambda_{49}$ and $\lambda_{50}$). By combining Lemma \ref{m7p12} with formula \eqref{Tay}, we infer that $\lambda_{27}$ (resp. $\lambda_{50}$) is a weak local minimizer, while $\lambda_{26}$ (resp. $\lambda_{49}$) is not.
\end{proof}

\begin{remark}
The positivity of $\lambda_h''$ is not enough for the optimality of weak local minimizers, and a necessary condition is the coercivity of the second order shape derivative, with respect to the $H^{1/2}$ norm of the deformation (see \cite{Dam, DL}).
Establishing the coercivity of $\lambda_h''$ turns out to be very complicated in general, due to the presence of the complex number $q$ and of the terms involving Bessel functions (see \eqref{derivataseconda}). Nevertheless, for some eigenvalues it is straightforward. Take for example $\lambda_3$: according to \eqref{derivataseconda} and Remark \ref{rem-qm}, we infer that the term multiplied by $q$ is non negative, thus
\begin{align*}
\lambda_3''(\D;V) &\geq  2 \lambda_3(\D) \left[ \sum_{\ell \in \mathbb Z } (\ell^2-1)|c_\ell|^2 + 2 \sum_{|\ell|\neq 1} \left( 2 + j_{1,1}  \frac{J_{\ell}'(j_{1,1})}{J_{\ell}(j_{1,1})}\right)|c_{1-\ell}|^2 \right] 
\\
&
= 2 \lambda_3(\D) \left[ \|\phi\|^2_{H^1} + 2 \sum_{k \geq 1}j_{1,1} \left( \frac{J_{2k+2}'(j_{1,1})}{J_{2k+2}(j_{1,1})}+\frac{J_{2k}'(j_{1,1})}{J_{2k}(j_{1,1})}\right)|c_{2k+1}|^2 \right] \geq  2 \lambda_3(\D) \|\phi\|^2_{H^1(\partial \D)}\,.
\end{align*}
Namely we have the $H^1$ coercivity (and hence the $L^\infty$ one) of $\lambda_3''$ at $\D$.
\end{remark}

\section{Some numerical results}\label{secnum}

\subsection{Numerical framework and optimization algorithm} We present in this section a numerical algorithm which can search for the shapes $\Omega$ which minimize $\lambda_h(\Omega)$ under constant width constraint. In Theorem \ref{wlm1} we prove that the disk is not a weak local minimizer for any of its simple eigenvalues. Moreover, in \ref{wlm2} it is proved that when the eigenvalue of the disk is double, only for a precise finite set of indices $h \geq 1$ the disk is a weak local minimizer for $\lambda_h(\Omega)$.  The computations presented below allow us to give further evidence that in these cases the disk is probably a global minimizer. Furthermore, for small enough indices $h$ it is possible to find shapes of given constant width which have their $h$-th eigenvalue smaller than the corresponding $h$-th eigenvalue for the disk.

The constant width constraint (or the diameter constraint) is difficult to handle numerically in optimization algorithms. One of the issues which appears when dealing with gradient based optimization algorithms is that admissible perturbations of the boundary which preserve the constant width property are not local. We refer to \cite{spheroforms}, \cite{BLOcw} and \cite{oudetCW} for methods of dealing with constant width constraint related to convex geometry. In \cite{BHcw12} the authors describe how to use the support function and its decomposition into Fourier series in order to study numerically optimization problems in the class of two dimensional shapes of constant width. It is this approach which inspired the method described below. Further applications of this method and extensions to higher dimensions are presented in \cite{ABcw}.

As was already noted in previous sections, if $f$ is the support function of a convex shape of constant width 2, then $f''(\theta)+f(\theta) \geq 0$ and $f(\theta)+f(\theta+\pi)=2$ for all $\theta \in [0,2\pi]$. 
Note that when writing the Fourier expansion of $f$
\[ f(\theta) = 1+\sum_{k=1}^\infty \left( a_k\cos(k\theta)+b_k \sin(k\theta)\right) \]
the constant width condition simply means that all coefficients with positive and even index must be equal to zero: $a_{2k} = b_{2k}=0$, for $k \geq 1$. 
Even though the constant width condition is simple to express in terms of the coefficients, we still need to impose the convexity condition
\begin{equation} f''(\theta)+f(\theta) = 1 + \sum_{k=1}^\infty \left( a_k(1-k^2)\cos(k\theta) + b_k(1-k^2)\sin(k\theta) \right) \geq 0, \ \forall \theta \in [0,2\pi].
\label{convex_ineq}
\end{equation}
In \cite{BHcw12} the authors provide an analytic characterization for the Fourier coefficients which satisfy this inequality. Using this analytic characterization leads to a semidefinite programming problem, which needs to be handled using specialized optimization software. Moreover, functionals considered in \cite{BHcw12} were always linear or quadratic in terms of the Fourier coefficients. Therefore, nonlinear functionals related to the Dirichlet Laplace eigenvalues $\lambda_h(\Omega)$ cannot be handled with the methods described in \cite{BHcw12}.

In order to deal with the convexity constraint, we choose instead to use a different method, described in \cite{antunesMM16}. Instead of searching for a global characterization of the Fourier coefficients of the support function of a convex set, we only impose the convexity constraint on a discretization of $[0,2\pi]$ fine enough. Indeed, let $\theta_1,...,\theta_M$ be a uniform discretization of $[0,2\pi]$, e.g., $\theta_i = \frac{2\pi i}{M}$, then condition \eqref{convex_ineq} can be replaced by  
\begin{equation}
1 + \sum_{k=1}^\infty \left( a_k(1-k^2)\cos(k\theta_i) + b_k(1-k^2)\sin(k\theta_i) \right) \geq 0, i=1,...,M.
\label{discrete_convex}
\end{equation}
This second formulation of the convexity constraint, which is weaker than \eqref{convex_ineq}, has the advantage of being linear in terms of the Fourier coefficients. This type of constraints can be implemented in many standard constrained optimization routines, like for example the \texttt{fmincon} function in the Matlab Optimization Toolbox.

In order to have a finite number of variables in our optimization, we only consider shapes which can be parametrized with Fourier coefficients up to rank $N$
\begin{equation}
 f(\theta) = 1+\sum_{k=1}^N \left( a_k\cos(k\theta)+b_k \sin(k\theta)\right).
 \label{finite_fourier}
 \end{equation}
We note that limiting the number of Fourier coefficient is not too restrictive, in the sense that when $N$ is large enough the class of shapes parametrized by support function given in \eqref{finite_fourier} can give a satisfactory approximation of any given shape. Moreover, one can repeat the optimization procedure for an increasingly higher number of coefficients, until we observe that the optimal shape does not change anymore and that the optimal value of the cost function does not improve.


%

In order to have an efficient optimization algorithm, we compute the derivatives of the eigenvalue in terms of the Fourier coefficients of the support function. To this aim, we first consider two types of perturbations, a cosine term and a sine term, namely two families of deformations {$\{V_k\}_{k}$ and $\{W_k\}_k$} as in Table \ref{perturb}.
\begin{table}[h]
\bgroup
\def\arraystretch{1.5}
\begin{tabular}{|c|p{65mm}|c|}
\hline
type & boundary perturbation & normal component
\\ 
\hline
 $\cos(k\theta)$ & ${V_k} = (\cos(k\theta)\cos \theta +k\sin(k\theta)\sin(\theta),$ \newline    $\cos(k\theta)\sin \theta -k\sin(k\theta)\cos(\theta))$     & ${V_k}.n = \cos(k\theta)$ \\ \hline     
   $\sin(k\theta)$ & ${W_k} = (\sin(k\theta)\cos \theta -k\cos(k\theta)\sin(\theta),$ \newline
      $\sin(k\theta)\sin \theta +k\cos(k\theta)\cos(\theta))$     & ${W_k}.n = \sin(k\theta)$ \\
 \hline 
\end{tabular} 
\caption{Transformation of perturbations of the support function into boundary perturbations and their corresponding normal components.}
\label{perturb}
\egroup
\end{table}

If $\lambda_h(\Omega)$ is simple, recalling the Hadamard formula \eqref{hadamard} for the first order shape derivative, and performing a change of variables (see \eqref{ds}), we have:
\begin{eqnarray*}
 \lambda_h'(\Omega; V_k) = -\int_0^{2\pi} (\partial_n u_h(x(\theta),y(\theta)))^2 \cos(k\theta)  (f''(\theta)+f(\theta))d\theta\,, \\
  \lambda_h'(\Omega;W_k ) = -\int_0^{2\pi} (\partial_n u_h(x(\theta),y(\theta)))^2 \sin(k\theta)  (f''(\theta)+f(\theta))d\theta \,,
\end{eqnarray*}
where $u_h$ is a normalized eigenfunction associated to $\lambda_h(\Omega)$.

%
%

In the case of double eigenvalues, difficulties may arise in the optimization (cf. \cite{oudeteigs}), since  $\lambda_h$ is not differentiable.
However, when using a precise solver in order to compute the eigenvalues and eigenfunctions, for example using spectral methods \cite{mpspack}, and a quasi-Newton method is used for optimization, these difficulties may be averted. We refer to the analysis in \cite{osting10} for further details.

The effective computation of the eigenvalues and eigenfunctions is done with the Matlab package MpsPack \cite{mpspack}. This software uses an accurate spectral method based on particular solutions. In our algorithm, we use the "ntd" method with $100$ basis functions. The computation of the integrals which represent the derivatives with respect to each Fourier coefficient is done using an order $1$ trapezoidal quadrature.

The optimization is done in Matlab using the \texttt{fmincon} procedure. This routine can perform the optimization of general functionals under matrix vector product equality/inequality constraints and non-linear equality/inequality constraints. As parameters we choose $N=40$, i.e. $80$ Fourier coefficients ($40$ sines and $40$ cosines). We use between $500$ and $1000$ points on the boundary where we impose the linear inequality constraints given by \eqref{discrete_convex}. The optimization algorithm is \emph{interior-point} with \emph{lbfgs} Hessian approximation. The constant width condition is imposed via a matrix product equality: all even coefficients of sine and cosine are zero. 

In order to avoid possible local minima, we choose a random vector of Fourier coefficients as initial condition and we project it onto the constraints, using again \texttt{fmincon} with a fictious objective function. We perform several optimizations starting each time from different random initializations in order to validate our results.

\subsection{Numerical results}

As noted in Theorems \ref{wlm1}, \ref{wlm2} there are only finitely many cases where the disk is a weak local minimizer. In each of these cases, the numerical simulations show that the disk is probably the global optimizer, since our algorithm did not manage to find better candidates. 

We run our algorithm for every value $h \leq 20$ and we note that the results are in accordance with the theoretical aspects recalled above. When the disk is not a weak local minimizer we manage to find shapes of fixed constant width $2$ which have a lower $h$-th eigenvalue than the corresponding one for the unit disk. We summarize these results in Figure \ref{num_results}.

The numerical simulations allow us to formulate some conjectures regarding the multiplicity of the optimal eigenvalues. We split the analysis in three cases: indices corresponding to a simple eigenvalue on the disk, first and second index in a pair of double eigenvalues for the disk.
\begin{itemize}[leftmargin=15pt]
\item $\lambda_h(\Bbb{D}))$ is simple: the numerical optimizer $\Omega$ has simple $h$-th eigenvalue
\item $\lambda_h(\Bbb{D}))=\lambda_{h+1}(\Bbb{D})$: the numerical optimizer $\Omega$ has simple $h$-th eigenvalue
\item $\lambda_{h-1}(\Bbb{D}))=\lambda_h(\Bbb{D})$: the numerical optimizer $\Omega$ has double $h$-th eigenvalue - $\lambda_{h-1}(\Omega) =\lambda_h(\Omega)$.  
\end{itemize}

We also remark that the numerical solutions we obtain for problem \eqref{minD} all have non-zero curvature radius, which means that they do not have singular points in their boundary, suggesting that they are $C^2$ regular.

\begin{table}
\centering
\begin{tabular}{ccc}
\includegraphics[width=0.25\textwidth]{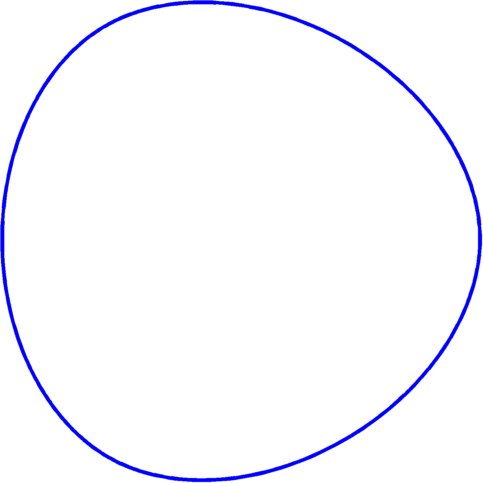} & 
\includegraphics[width=0.25\textwidth]{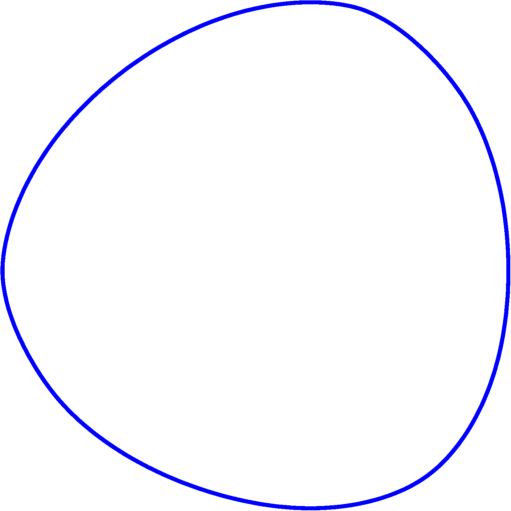} &
\includegraphics[width=0.25\textwidth]{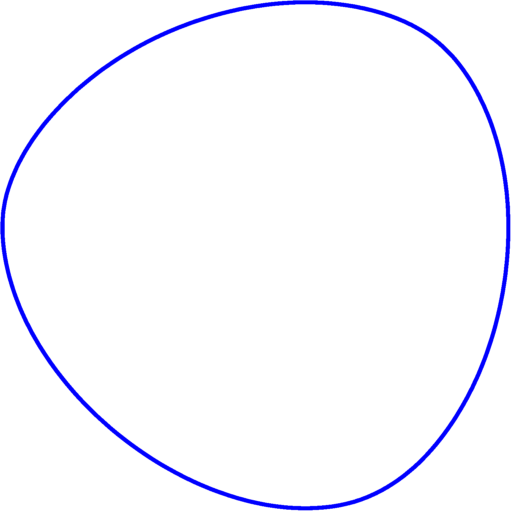} \\
$\lambda_6(\Omega)=30.453$ & 
$\lambda_9(\Omega)=49.080$ & 
$\lambda_{10}(\Omega)=49.084$ \\
$\lambda_6(\Bbb{D}) = 30.4713$ & 
$\lambda_9(\Bbb{D}) = 49.2184$ &
$\lambda_{10}(\Bbb{D}) = 49.2184$\\
\includegraphics[width=0.25\textwidth]{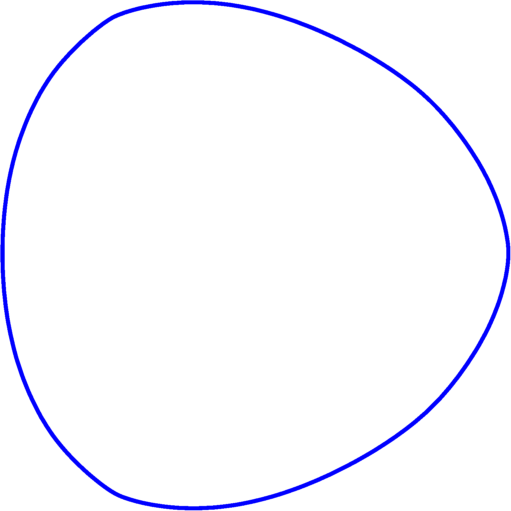} & 
\includegraphics[width=0.25\textwidth]{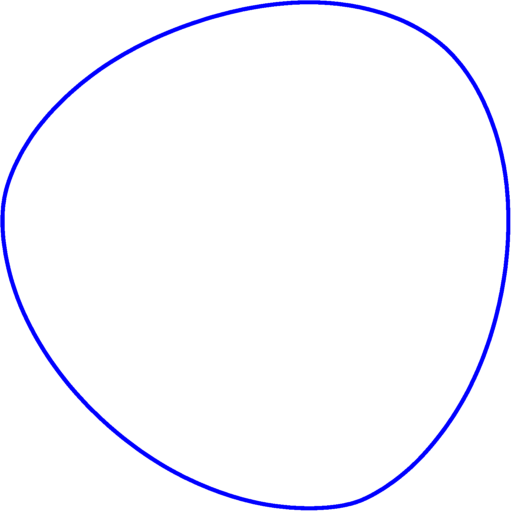} &
\includegraphics[width=0.25\textwidth]{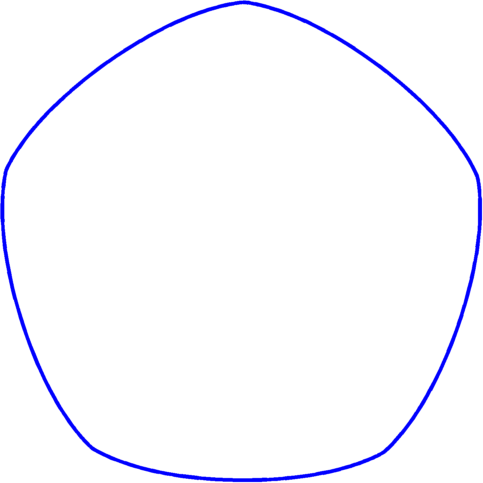} \\
$\lambda_{13}(\Omega)=70.222$ & 
$\lambda_{14}(\Omega)=70.244$ & 
$\lambda_{15}(\Omega)=73.589$ \\
$\lambda_{13}(\Bbb{D}) = 70.8499$ & 
$\lambda_{14}(\Bbb{D}) = 70.8499$ &
$\lambda_{15}(\Bbb{D}) = 74.8868$\\
\includegraphics[width=0.25\textwidth]{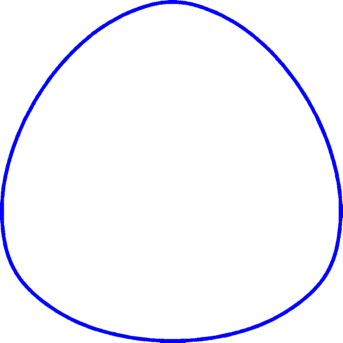} & 
\includegraphics[width=0.25\textwidth]{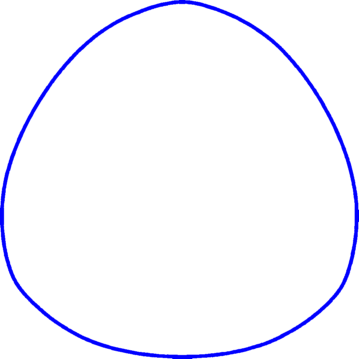} &
\includegraphics[width=0.25\textwidth]{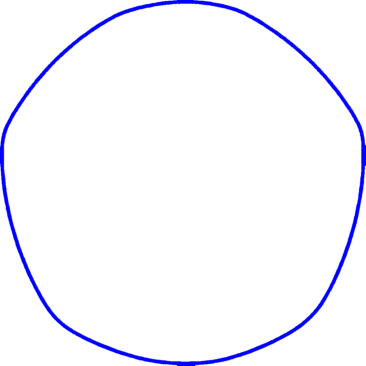} \\
$\lambda_{18}(\Omega)=93.626$ & 
$\lambda_{19}(\Omega)=93.683$ & 
$\lambda_{20}(\Omega)=98.254$ \\
$\lambda_{13}(\Bbb{D}) = 95.2776$ & 
$\lambda_{19}(\Bbb{D}) = 95.2776$ &
$\lambda_{20}(\Bbb{D}) = 98.7263$
\end{tabular}
\caption{Non circular shapes of width $2$ which are candidates to be the minimizers of $\lambda_h$, for $h=6,9,10,13,14,15,18,20$.}
\label{num_results}
\end{table}


\medskip

Beniamin \textsc{Bogosel}, CNRS, Centre de Math\'ematiques Appliqu\'ees, \'Ecole Polytechnique, \\ email: \texttt{beniamin.bogosel@cmap.polytechnique.fr}

Antoine \textsc{Henrot}, Institut \'Elie Cartan de Lorraine, UMR 7502, Universit\'e de Lorraine CNRS, email: \texttt{antoine.henrot@univ-lorraine.fr} 

Ilaria \textsc{Lucardesi}, Institut \'Elie Cartan de Lorraine, UMR 7502, Universit\'e de Lorraine CNRS, email: \texttt{ilaria.lucardesi@univ-lorraine.fr}


\begin{thebibliography}{99}

\bibitem{antunesMM16}
{\sc P.~R.~S.~Antunes:}
\newblock Maximal and minimal norm of Laplacian eigenfunctions in a given
  subdomain.
\newblock {\em Inverse Problems} {\bf 32}, no. 11, (2016)

\bibitem{ABcw}
{\sc P.~R.~S.~Antunes, B.~Bogosel:}
\newblock Parametric shape optimization using the support functions,
\newblock {(2018)}

\bibitem{mpspack}
{\sc A.~H.~Barnett, T.~Betcke:}
\newblock Stability and convergence of the method of fundamental solutions for
  {H}elmholtz problems on analytic domains.
\newblock {\em J. Comput. Phys.} {\bf 227}, no. 14, 7003--7026 (2008)

\bibitem{BHcw12}
{\sc T.~Bayen, D.~Henrion:}
\newblock Semidefinite programming for optimizing convex bodies under width
  constraints.
\newblock {\em Optim. Methods Softw.} {\bf 27}, no. 6, 1073--1099 (2012)

\bibitem{BLOcw}
{\sc T.~Bayen, T.~Lachand-Robert, E.~Oudet:}
\newblock Analytic parametrization of three-dimensional bodies of constant
  width.
\newblock {\em Arch. Ration. Mech. Anal.} {\bf 186}, no. 2, 225--249 (2007)


\bibitem{Ber} {\sc A.~Berger:} {The eigenvalues of the Laplacian with Dirichlet boundary condition in $\mathbb{R}^2$ are
almost never minimized by disks}, {\em Ann. Global Anal. Geom.} {\bf 47}, no. 3, 285--304  (2015)

\bibitem{bogreg} {\sc B.~Bogosel:} {Regularity result for a shape optimization problem under perimeter constraint}, {\em Comm. Anal. Geom.} (to appear, 2017)

\bibitem{Buc} {\sc D.~Bucur:} {Minimization of the k-th eigenvalue of the Dirichlet Laplacian}, {\em Arch. Ration. Mech. Anal.} {\bf 206}, no. 3, 1073--1083  (2012)

\bibitem{bubuhe}
{\sc D.~Bucur, G.~Buttazzo, A.~Henrot:} {Minimization of $\lambda_2(\Omega)$ with a perimeter constraint}, {\em Indiana Univ. Math. J.}  {\bf 58}, 2709--2728 (2009)

\bibitem{BMPV} {\sc  D.~Bucur, D.~Mazzoleni, A.~Pratelli, B.~Velichkov:} {Lipschitz regularity of the eigenfunctions on optimal domains}, {\em Arch. Ration. Mech. Anal.} {\bf 216}, no. 1, 117--151 (2015)

\bibitem{colesanti}  {\sc  A.~Colesanti:} {Brunn-{M}inkowski inequalities for variational functionals and  related problems}, {\em Adv. Math.} {\bf 194}, no. 1, 105--140 (2005)

\bibitem{Dam}{\sc M.~Dambrine:} {On variations of the shape Hessian and sufficient conditions for the stability of critical shapes}, {\em Rev. R. Acad. Cienc. Exactas F\'is. Nat.}, Ser. A Mat. {\bf 96}, no. 1, 95--121 (2002) 

\bibitem{DL}{\sc M.~Dambrine, J.~Lamboley:} {Stability in shape optimization with second variation} \url{https://arxiv.org/abs/1410.2586} (last version 2016)

\bibitem{dP-V} {\sc G.~De Philippis, B.~Velichkov:}
{Existence and regularity of minimizers for some spectral functionals with perimeter constraint}, {\em Appl. Math. Optim.} {\bf 69}, no. 2, 199--231 (2014)

\bibitem{Fab} {\sc G.~Faber:}  {Beweis, dass unter allen homogenen
Membranen von gleicher Fl\"{a}che und gleicher Spannung die
kreisf\"{o}rmige den tiefsten Grundton gibt }, {\em Sitz. Ber. Bayer. Akad. Wiss.}, 169--172 (1923)

\bibitem{H1}{\sc A.~Henrot:} Extremum problems for eigenvalues of elliptic operators. Birkh\"auser, Basel (2006).

\bibitem{H2}{\sc A.~Henrot (ed):} Shape Optimization and Spectral Theory. De Gruyter open (2017), freely downloadable at
https://www.degruyter.com/view/product/490255

\bibitem{HP}{\sc A.~Henrot, M.~Pierre:} Variation et Optimisation de Formes. Une Analyse
G\'eom\'etrique. Math\'ematiques \& Applications 48. Springer, Berlin (2005)

\bibitem{Kra1} {\sc  E. Krahn:} {\"Uber eine von Rayleigh formulierte
Minimaleigenschaft des Kreises\/}, {\em Math. Ann.} {\bf 94}, 97--100 (1924)

\bibitem{Kr}{\sc I.~Krasikov:} {Approximations for the Bessel and Airy functions with an explicit error term}, {\em LMS J. Comput. Math.} \textbf{17}, no. 1, 209--225  (2014), doi:10.1112/S1461157013000351

\bibitem{spheroforms}
{\sc T.~Lachand-Robert, E.~Oudet:}
\newblock Bodies of constant width in arbitrary dimension,
\newblock {\em Math. Nachr.} {\bf 280}, no. 7, 740--750 (2007)

\bibitem{Maz-Pra} {\sc D. Mazzoleni, A. Pratelli:} {Existence of minimizers for spectral problems}, {\em J. Math. Pures Appl.} {\bf 9}, no. 3, 433--453 (2013)

\bibitem{osting10}
{\sc B.~Osting:}
\newblock Optimization of spectral functions of {D}irichlet-{L}aplacian eigenvalues,
\newblock {\em Discrete Comput. Geom.} {\bf 49}, no. 2, 411--428 (2013)

\bibitem{oudeteigs}
{\sc E.~Oudet:}
\newblock Numerical minimization of eigenmodes of a membrane with respect to the domain,
\newblock {\em J. Comput. Phys.} {\bf 229}, no. 22, 411--428 (2013)

\bibitem{oudetCW}
{\sc E.~Oudet:}
\newblock Shape optimization under width constraint,
\newblock {\em ESAIM Control Optim. Calc. Var.} {\bf 10}, no. 3, 315--330 (2004)

\bibitem{Sch} {\sc R. Schneider:} Convex bodies: the Brunn-Minkowski theory, {\em Encyclopedia of Mathematics and its Applications}, 151. Cambridge University Press, Cambridge (2014)

\end{thebibliography}
\end{document}